\documentclass[oneside]{amsart}
\usepackage[T1]{fontenc}
\usepackage[latin9]{inputenc}
\usepackage{amsthm}
\usepackage{amssymb}
\usepackage{esint}

\makeatletter
\numberwithin{equation}{section}
\numberwithin{figure}{section}
\theoremstyle{plain}
\newtheorem{thm}{\protect\theoremname}
  \theoremstyle{plain}
  \newtheorem{prop}[thm]{\protect\propositionname}
  \theoremstyle{plain}
  \newtheorem{lem}[thm]{\protect\lemmaname}

\usepackage{url}

\makeatother

  \providecommand{\lemmaname}{Lemma}
  \providecommand{\propositionname}{Proposition}
\providecommand{\theoremname}{Theorem}

\begin{document}
\address[Minoru Hirose]{Multiple Zeta Research Center, Kyushu University}
\email{m-hirose@math.kyushu-u.ac.jp}
\address[Nobuo Sato]{Department of Mathematics, Kyoto University}
\email{saton@math.kyoto-u.ac.jp}

\title{On Hoffman's conjectural identity}

\author{Minoru Hirose, Nobuo Sato}
\begin{abstract}
In this paper, we shall prove the equality 
\[
\zeta(3,\{2\}^{n},1,2)=\zeta(\{2\}^{n+3})+2\zeta(3,3,\{2\}^{n})
\]
conjectured by Hoffman using certain identities among iterated integrals
on $\mathbb{P}^{1}\setminus\{0,1,\infty,z\}$.
\end{abstract}
\maketitle

\section{Introduction}

Multiple zeta values, or MZVs in short, are real numbers defined by
\[
\zeta(k_{1},\ldots,k_{d})=\sum_{0<n_{1}<\cdots<n_{d}}\frac{1}{n_{1}^{k_{1}}\cdots n_{d}^{k_{d}}}.
\]
In February 2000, Hoffman proposed the following conjectural formula
on his homepage \cite{HH}:
\[
\zeta(3,\{2\}^{n},1,2)=\zeta(\{2\}^{n+3})+2\zeta(3,3,\{2\}^{n})\qquad(n\in\mathbb{Z}_{\geq0}).
\]
Here, we used the simplified notations, e.g. $\zeta(3,\{2\}^{n},1,2)=\zeta(3,\overbrace{2,\ldots,2}^{n\,\mathrm{times}},1,2).$
In \cite{HH}, Hoffman also mentioned that J. Vermaseren had checked
his conjecture up to $n=8$ by using the MZV data mine presented in
\cite{DataMine}. In \cite{Ch16theses} or \cite{Ch16arxiv}, Charlton
proved the existence of a rational number $q_{n}$ such that
\[
\zeta(3,\{2\}^{n},1,2)=q_{n}\zeta(\{2\}^{n+3})+2\zeta(3,3,\{2\}^{n})\qquad(n\in\mathbb{Z}_{\geq0})
\]
by calculating Brown\textquoteright{}s operators $D_{r}$'s. However,
the method based on Brown's operators does not seem to provide a way
to determine the value of $q_{n}$. Thus, in this paper we give a
proof of Hoffman's conjecture by using a completely different approach
based on the iterated integrals on $\mathbb{P}^{1}\setminus\{0,1,\infty,z\}$.
More generally, we shall prove the following theorem.
\begin{thm}
\label{Main_theorem}For $m,s\in\mathbb{Z}_{\geq1}$ and $n\in\mathbb{Z}_{\geq0}$,
\begin{align*}
 & \zeta(\{2\}^{m-1},3,\{2\}^{n},1,\{2\}^{s})\\
 & =\zeta(\{2\}^{n+m+s+1})+\zeta(\{2\}^{s-1},3,\{2\}^{m-1},3,\{2\}^{n})+\zeta(\{2\}^{m-1},3,\{2\}^{s-1},3,\{2\}^{n}).
\end{align*}

\end{thm}
This theorem together with the duality relation gives a special case
of Charlton's generalized cyclic insertion conjecture (see \cite[Section 2]{Ch16theses}
or \cite{Ch16arxiv}). In particular, the case $m=s=1$ of the theorem
gives Hoffman's conjecture.

In Section \ref{sec:review_of_iterated_integral}, we briefly review
the iterated integrals on $\mathbb{P}^{1}\setminus\{0,1,\infty,z\}$
which was investigated in \cite{HIST} by Iwaki, Tasaka and the authors.
In Section \ref{sec:Proof_of_Hoffman's_conjecture}, we give a proof
of Theorem \ref{Main_theorem} using a lemma in Section \ref{sec:review_of_iterated_integral}. 

The authors thank Michael E. Hoffman for kindly letting us know when
the conjecture was formulated.

\section{A review of iterated integrals on $\mathbb{P}^{1}\setminus\{0,1,\infty,z\}$\label{sec:review_of_iterated_integral}}

In this section, we review the differential formula for the iterated
integrals on $\mathbb{P}^{1}\setminus\{0,1,\infty,z\}$ given in \cite{HIST}.
Let $\mathcal{A}:=\mathbb{Q}\left\langle e_{0},e_{1},e_{z}\right\rangle $
denote the non-commutative polynomial algebra generated by three indeterminates
$e_{0},e_{1}$ and $e_{z}$ over $\mathbb{Q}$. Let $\mathcal{A}^{1}$
denote the subalgebra of $\mathcal{A}$ defined by 
\[
\mathcal{A}^{1}:=\mathbb{Q}+\mathbb{Q}e_{z}+e_{1}\mathcal{A}e_{0}+e_{z}\mathcal{A}e_{0}+e_{1}\mathcal{A}e_{z}+e_{z}\mathcal{A}e_{z}.
\]
We assume that $z\in\mathbb{C}\setminus[0,1]$ and define a linear
function $L$ on $\mathcal{A}^{1}$, by assigning a word $w=e_{a_{1}}\cdots e_{a_{m}}$
in $e_{0},e_{1}$ and $e_{z}$, the iterated integral 
\[
I(0;a_{1},\ldots,a_{m};1)=\int_{0<t_{1}<\cdots<t_{m}<1}\prod_{i=1}^{m}\frac{dt_{i}}{t_{i}-a_{i}}.
\]
Here, the path of integration is taken as a line segment between $0$
and $1$. Note that, for $w\in\mathcal{A}^{1}$, $L(w)$ defines a
holomorphic function of $z$ on $\mathbb{C}\setminus[0,1]$. The MZVs
are, then, expressed as 
\[
\zeta(k_{1},\ldots,k_{d})=(-1)^{d}L(e_{1}e_{0}^{k_{1}-1}\cdots e_{1}e_{0}^{k_{d}-1}).
\]

Let $\partial_{z,0}$ and $\partial_{z,1}$ be linear operators on
$\mathcal{A}^{1}$ defined by
\[
\partial_{z,b}\left(e_{a_{1}}\cdots e_{a_{n}}\right):=\sum_{i=1}^{n}\left(\delta_{\{a_{i},a_{i+1}\},\{z,b\}}-\delta_{\{a_{i-1},a_{i}\},\{z,b\}}\right)e_{a_{1}}\cdots\widehat{e_{a_{i}}}\cdots e_{a_{n}}
\]
for $b\in\{0,1\}$. Here, we set $a_{0}=0,a_{n+1}=1$ and $\delta_{\{a,b\},\{c,d\}}$
denotes the Kronecker delta, i.e. 
\[
\delta_{\{a,b\},\{c,d\}}=\begin{cases}
\:1 & \mbox{ if }\{a,b\}=\{c,d\}\\
\:0 & \mbox{ if }\{a,b\}\neq\{c,d\}.
\end{cases}
\]
Note that $\partial_{z,0}$ and $\partial_{z,1}$ reduce the word
length by $1$. The following is a fundamental property of $\partial_{z,0}$
and $\partial_{z,1}$.
\begin{prop}
[\cite{HIST}]\label{prop:differential_formula}For $w\in\mathcal{A}^{1}$,
\[
\frac{d}{dz}L(w)=\frac{1}{z}L(\partial_{z,0}w)+\frac{1}{z-1}L(\partial_{z,1}w).
\]

\end{prop}
Let $\mathcal{I}:=\mathcal{A}e_{z}\mathcal{A}\subset\mathcal{A}$
be the two-sided ideal generated by $e_{z}$. The following lemma
is useful.
\begin{lem}
\label{lem:diff_zero}If $w\in\mathcal{I}\cap\mathcal{A}^{1}$ and
$L(\partial_{z,0}w)=L(\partial_{z,1}w)=0$, then, $L(w)=0.$\end{lem}
\begin{proof}
Since $L(\partial_{z,0}w)=L(\partial_{z,1}w)=0$, $\frac{d}{dz}L(w)=0$
by Proposition \ref{prop:differential_formula} and so $L(w)$ is
a constant. Since $w\in\mathcal{I}\cap\mathcal{A}^{1}$, $\lim_{z\rightarrow\infty}\left(L(w)\right)=0$,
which means the constant is zero.
\end{proof}

\section{A proof of Hoffman's conjecture\label{sec:Proof_of_Hoffman's_conjecture}}

In this section, we prove Theorem \ref{Main_theorem}. Fix $s\in\mathbb{Z}_{>0}$.
We set
\[
A_{n}=(e_{1}e_{0})^{n},\, B_{n}=(e_{1}e_{0})^{n}e_{1},\,\overline{A}_{n}=(e_{0}e_{1})^{n},\,\overline{B}_{n}=(e_{0}e_{1})^{n}e_{0}
\]
and
\begin{align*}
F_{ee}(m,n):= & \; A_{s}\overline{B}_{m}e_{z}A_{n}+A_{m}e_{z}(A_{s}+\overline{A}_{s})e_{z}A_{n}+A_{m}e_{z}B_{n}A_{s}\\
 & \quad-B_{s+m}e_{z}A_{n}-A_{m}e_{z}\overline{B}_{s+n}\\
F_{oe}(m,n):= & \; A_{s}\overline{A}_{m+1}e_{z}A_{n}+B_{m}e_{z}(A_{s}+\overline{A}_{s})e_{z}A_{n}+B_{m}e_{z}B_{n}A_{s}\\
 & \quad-A_{s+m+1}e_{z}A_{n}-B_{m}e_{z}\overline{B}_{s+n}\\
F_{eo}(m,n):= & \; A_{s}\overline{B}_{m}e_{z}\overline{B}_{n}+A_{m}e_{z}(A_{s}+\overline{A}_{s})e_{z}\overline{B}_{n}+A_{m}e_{z}\overline{A}_{n+1}A_{s}\\
 & \quad-B_{s+m}e_{z}\overline{B}_{n}-A_{m}e_{z}A_{s+n+1}\\
F_{oo}(m,n):= & \; A_{s}\overline{A}_{m+1}e_{z}\overline{B}_{n}+B_{m}e_{z}(A_{s}+\overline{A}_{s})e_{z}\overline{B}_{n}+B_{m}e_{z}\overline{A}_{n+1}A_{s}\\
 & \quad-A_{s+m+1}e_{z}\overline{B}_{n}-B_{m}e_{z}A_{s+n+1}
\end{align*}
for $m,n\in\mathbb{Z}_{\geq0}$. Additionally, we define 
\[
F_{ee}(m,n)=F_{oe}(m,n)=F_{eo}(m,n)=F_{oo}(m,n)=0
\]
if $m<0$ or $n<0$. 
\begin{prop}
\label{prop:1}For $m,n\in\mathbb{Z}_{\geq0}$, we have
\[
L(F_{ee}(m,n))=L(F_{oe}(m,n))=L(F_{eo}(m,n))=L(F_{oo}(m,n))=0.
\]
\end{prop}
\begin{proof}
We shall prove the proposition by induction on the word length. For
$w\in\mathcal{A}^{1}\cap\mathbb{Q}\left\langle e_{0},e_{1}\right\rangle $,
we set 
\[
\Delta(w)=w-\tau(w),
\]
where $\tau$ is an anti-automorphism defined by $\tau(e_{0})=-e_{1}$
and $\tau(e_{1})=-e_{0}$. Since
\begin{align*}
\partial_{z,0}\left(A_{s}\overline{B}_{m}e_{z}A_{n}\right)= & \; A_{s}\overline{A}_{m}e_{z}A_{n}-A_{s}\overline{B}_{m+n},\\
\partial_{z,0}\left(A_{m}e_{z}(A_{s}+\overline{A}_{s})e_{z}A_{n}\right)= & \begin{cases}
\begin{aligned} & B_{m-1}e_{z}(A_{s}+\overline{A}_{s})e_{z}A_{n}\\
 & \;-A_{s+m}e_{z}A_{n}-A_{m}e_{z}A_{s+n}
\end{aligned}
 & \mbox{ for }m>0\\
-A_{s}e_{z}A_{n}-e_{z}A_{s+n} & \mbox{ for }m=0,
\end{cases}\\
\partial_{z,0}\left(A_{m}e_{z}B_{n}A_{s}\right)= & \begin{cases}
B_{m-1}e_{z}B_{n}A_{s}+\tau(A_{s}\overline{B}_{m+n}) & \mbox{ for }m>0\\
\tau(A_{s}\overline{B}_{n}) & \mbox{ for }m=0,
\end{cases}\\
\partial_{z,0}\left(-B_{s+m}e_{z}A_{n}\right)= & \;0,\\
\partial_{z,0}\left(-A_{m}e_{z}\overline{B}_{s+n}\right)= & \begin{cases}
-B_{m-1}e_{z}\overline{B}_{s+n}+A_{m}e_{z}A_{s+n} & \mbox{ for }m>0\\
e_{z}A_{s+n} & \mbox{ for }m=0,
\end{cases}
\end{align*}
we find 
\[
\partial_{z,0}F_{ee}(m,n)=F_{oe}(m-1,n)-\Delta(A_{s}\overline{B}_{m+n})
\]
for $m,n\in\mathbb{Z}_{\geq0}$. Similarly, 
\begin{align*}
\partial_{z,1}F_{ee}(m,n) & =-F_{eo}(m,n-1)+\Delta(A_{s}\overline{B}_{m+n}),\\
\partial_{z,0}F_{oe}(m,n) & =0,\\
\partial_{z,1}F_{oe}(m,n) & =F_{ee}(m,n)-F_{oo}(m,n-1),\\
\partial_{z,0}F_{eo}(m,n) & =-F_{ee}(m,n)+F_{oo}(m-1,n),\\
\partial_{z,1}F_{eo}(m,n) & =0,\\
\partial_{z,0}F_{oo}(m,n) & =-F_{oe}(m,n)+\Delta(A_{s}\overline{B}_{m+n+1}),\\
\partial_{z,1}F_{oo}(m,n) & =F_{eo}(m,n)-\Delta(A_{s}\overline{B}_{m+n+1}).
\end{align*}
Using Lemma \ref{lem:diff_zero} and the duality relation, i.e. $L(\Delta(w))=0$
for $w\in\mathcal{A}^{1}\cap\mathbb{Q}\left\langle e_{0},e_{1}\right\rangle $,
we obtain the proposition.
\end{proof}
By Proposition \ref{prop:1}, it follows that
\begin{align*}
0 & =\lim_{z\rightarrow-0}L(F_{ee}(m,n))\\
 & =L(A_{s}\overline{B}_{m}\overline{B}_{n}+A_{m}\overline{B}_{s}\overline{B}_{n}+A_{m}\overline{A}_{n+1}A_{s}-A_{s+m+n+1})\\
 & =(-1)^{s+m+n}\left\{ \zeta(\{2\}^{s-1},3,\{2\}^{m-1},3,\{2\}^{n})+\zeta(\{2\}^{m-1},3,\{2\}^{s-1},3,\{2\}^{n})\right.\\
 & \left.\qquad\qquad\qquad\qquad\qquad\qquad-\zeta(\{2\}^{m-1},3,\{2\}^{n},1,\{2\}^{s})+\zeta(\{2\}^{s+m+n+1})\right\} ,
\end{align*}
which proves Theorem \ref{Main_theorem}.

\bibliographystyle{abbrv}
\bibliography{HoffmanConj}

\end{document}